\documentclass[12pt,a4paper]{amsart}

\usepackage{amssymb}
\usepackage{ifthen}
\usepackage{graphicx}
\usepackage{color}
\dedicatory{}
\commby{}
\theoremstyle{plain}
\newtheorem{theorem}[equation]{Theorem}

\newtheorem{lemma}[equation]{Lemma}
\newtheorem{example}[equation]{Example}

\theoremstyle{definition}
\newtheorem{definition}[equation]{Definition}
\theoremstyle{remark}

\newtheorem{nonsec}[equation]{}

\numberwithin{equation}{section}
\newcommand{\diam}{\ensuremath{\textrm{diam}}}
\newcommand{\beq}{\begin{equation}}
\newcommand{\eeq}{\end{equation}}
\newcommand{\ben}{\begin{enumerate}}
\newcommand{\een}{\end{enumerate}}
\newcommand{\bequu}{\begin{eqnarray*}}
\newcommand{\eequu}{\end{eqnarray*}}
\newcommand{\bequ}{\begin{eqnarray}}
\newcommand{\eequ}{\end{eqnarray}}

\DeclareMathOperator{\dist}{dist}

\newcounter{minutes}
\divide\time by 60
\newcounter{hours}
\multiply\time by 60 \addtocounter{minutes}{-\time}

\renewcommand{\thefootnote}{\number_style{footnote}}
\begin{document}

\def\thefootnote{}
\footnotetext{ \texttt{\tiny File:~\jobname .tex,
           printed: \number\year-\number\month-\number\day,
           \thehours.\ifnum\theminutes<10{0}\fi\theminutes}
} \makeatletter\def\thefootnote{\@arabic\c@footnote}\makeatother

\title{Apollonian metric, uniformity and Gromov hyperbolicity}

\author{Yaxiang  Li}
\address{Yaxiang Li,  Department of Mathematics, Hunan First Normal University, Changsha,
Hunan 410205, P.R.China}\address{Hunan Provincial Key Laboratory of Mathematical Modeling and Analysis in Engineering, Changsha University of Science and Technology, Changsha 410114, Hunan, P.R.China}
\email{yaxiangli@163.com}

\author[]{Matti Vuorinen}
\address{Matti Vuorinen, Department of Mathematics and Statistics, University of Turku,
FIN-20014 Turku, Finland}
\email{vuorinen@utu.fi}

\author{Qingshan Zhou${}^{\mathbf{*}}$}
\address{Qingshan Zhou, school of mathematics and big data, foshan university,  Foshan, Guangdong 528000, People's Republic
of China} \email{q476308142@qq.com}

\date{}
 \keywords{
Roughly Apollonian bilipschitz mappings, uniform domains, Apollonian metric, Gromov hyperbolicity.\\
${}^{\mathbf{*}}$ Corresponding author}

\begin{abstract}The main purpose of this paper is to investigate the properties of a mapping which is required to be  roughly bilipschitz with respect to the Apollonian metric (roughly Apollonian bilipschitz) of its domain. We prove that under these mappings
 the uniformity, $\varphi$-uniformity and $\delta$-hyperbolicity (in the sense of Gromov with respect to quasihyperbolic metric) of proper domains of $\mathbb{R}^n$ are invariant.  As applications, we give four equivalent conditions for a quasiconformal
mapping which is defined on  a uniform domain to be roughly Apollonian
bilipschitz, and we conclude that $\varphi$-uniformity is invariant under quasim\"obius mappings.
\end{abstract}


\thanks{The research was partly supported  by NNSF of
China (Nos. 11601529, 11671127) and Hunan Provincial Key Laboratory of Mathematical Modeling and Analysis in Engineering ( No.2017TP1017).}

\maketitle{} \pagestyle{myheadings} \markboth{}{Apollonian metric, uniformity and Gromov hyperbolicity}

\section{Introduction and main results}\label{sec-1}
In geometric function theory, one mainly investigates the interplay between analytic properties of mappings and geometric properties of sets and domains. A key question is how to measure the
distance between two points $x,y$ in a proper subdomain $G \subset {\mathbb R}^n\,.$ Instead of
using distance functions which measure the position of the points with respect to each other, such as Euclidean and chordal metrics,  it is more useful to take into account also
 the position of the points with respect to the boundary of the domain. Many authors have used this idea  to define metrics of hyperbolic type and to study the geometries defined by these metrics in domains. Some examples are
 the quasihyperbolic metric, Apollonian metric, the distance ratio metric, Seittenranta's metric, see \cite{Be, GH, GO, H2, S, V2}. In particular, the quasihyperbolic metric has become a basic tool in geometric function theory and it has many important applications \cite{GH,Vai1}.

Suppose that we are given a domain $G \subset {\mathbb R}^n$ and two metrics $m_1$ and $m_2$ on it. It is natural to study whether or not these metrics are comparable in some sense. It turns out that the comparison properties of metrics imply geometric properties of the domain: this idea was used by Gehring and Osgood \cite{GO} to characterise
so called uniform domains, by Gehring and Hag \cite{GH1} to study quasidisks, by Vuorinen \cite{V2} to define $\varphi$-uniform domains,
by H\"ast\"o \cite{H1} to study comparison properties of so called Apollonian metric.    Seittenranta \cite{S} defined a M\"obius invariant
metric on subdomains of ${\mathbb{R}^n }$ and, comparing this metric to Ferrand's metric, defined a M\"obius invariant class of domains. In the general case, we could call domains with such a comparison property $(m_1,m_2)$-uniform domains. Uniform domains and quasidisks form
 classes of domains, which have been studied by many authors. In spite of all this work, there are many pairs of function theoretically interesting metrics $m_1,m_2$, for
which practically nothing is known about $(m_1,m_2)$-uniform domains.

One of the key features of hyperbolic type metrics is the Gromov hyperbolicity property. It is well-known that the Gehring-Osgood $\widetilde{j}$-metric and the quasihyperbolic metric of uniform domains are Gromov hyperbolic. We note that H\"{a}st\"{o} in \cite{H3} proved that the $\widetilde{j}$-metric is always Gromov hyperbolic, but the  $j$-metric  is Gromov hyperbolic if and only if $G$ has exactly one boundary point. In fact, in $\mathbb{R}^n$, many results in quasiconformal mappings can be explained through negative curvature, or ``Gromov hyperbolicity". It would be interesting to know, what the precise relationship between the higher dimensional quasiconformal theory and the work of Gromov is. On the other hand, Gromov hyperbolicity for metric spaces is a coarse notion of negative curvature which yields a very satisfactory theory. It is natural to consider the properties of coarse maps with respect to the hyperbolic type metrics and the geometry of domains. In the spirit of this motivation, we mainly study a class of mappings which are roughly bilipschitz with respect to the Apollonian metric in $\mathbb{R}^n$.

In fact, the study of Apollonian metric and  the so called Apollonian bilipschitz mapping, (i.e., bilipschitz mapping with respect to Apollonian metric)  has been largely motivated and considered by questions about Apollonian isometries, which in turn was a continuation of work by Beardon \cite{Be}, Gehring and Hag \cite{GH}, H\"{a}st\"{o} and his  collaborators \cite{H1,H2, HK, HI}. In order to make this paper more readable, we review some notations from \cite{V2} and \cite{H1}.

We will consider domains (open connected non-empty sets) $G$  in the M\"{o}bius space $\overline{\mathbb{R}^n}=\mathbb{R}^n\cup \{\infty\}$. The {\it Apollonian metric} is defined by
$$\alpha_G(x,y):=\log\sup_{a,b\in \partial G}|a,y,x,b| \;\;\;\;\;\;\;\;\;\; \mbox{where}\;|a,y,x,b|=\frac{|a-x||b-y|}{|a-y||b-x|},$$
for $x,y\in G\subsetneq \mathbb{R}^n$ with understanding that if $a=\infty$ then we set $|a-x|/|a-y|=1$ and similarly for $b$. It is in fact a metric if   $\partial G$ is not contained in a hyperplane or sphere, as was noted by \cite[Theorem $1.1$]{Be}.

In the paper \cite{GH} Gehring and Hag proved that a quasi-disk is invariant under a quasiconformal mapping which is also  Apollonian bilipschitz. Along this line, H\"{a}st\"{o} \cite{H1}  introduced  $A$-uniform domains:  A domain $G\subsetneq \mathbb{R}^n$ is said to be {\it $A$-uniform  with constant $A_1$ if  for some constant $A_1>0$ and for every $x,y$ $\in G$, we have $k_G(x,y)\leq A_1 \alpha_G(x,y)$, where $k_G(x,y)$ is the quasihyperbolic metric (for definition see Subsection 2.2) between $x$ and $y$ in $G$. A domain $G\subsetneq \mathbb{R}^n$ is said to be  $A$-uniform if it is $A$-uniform with some constant $A_1<\infty$.} In particular, he proved the following result:

\begin{theorem}\label{z-1}$($\cite[Theorem $1.8$]{H1}$)$ Let $G\subsetneq \mathbb{R}^n$ be $A$-uniform and let $f:G\to G' \subsetneq\mathbb{R}^n$ be an Apollonian bilipschitz mapping. The following conditions are equivalent:
\begin{enumerate}
\item $G'$ is $A$-uniform;
\item $f$ is quasiconformal in $G$.
\end{enumerate}
\end{theorem}

We note that H\"{a}st\"{o} in \cite[Proposition 6.6]{H1} proved that a domain $G$ is $A$-uniform if and only if $G$ is $L$-quasi-isotropic (for definition see Subsection 2.19) and $\alpha_G$ is quasiconvex. So in this paper, we first complement Theorem \ref{z-1} in the following way.

\begin{theorem}\label{cor}
Let $G\subsetneq \mathbb{R}^n$ be $A$-uniform and let $f:G\to G' \subsetneq\mathbb{R}^n$ be an Apollonian bilipschitz mapping. Then the following conditions are equivalent:
\begin{enumerate}
  \item $G'$ is $A$-uniform;
  \item $G'$ is $L$-quasi-isotropic;
  \item $f$ is quasiconformal in $G$;
  \item $f$ is quasim\"{o}bius in $G$.
\end{enumerate}
\end{theorem}

Furthermore, it follows from \cite[Example $4.4$ and Proposition $6.6$]{H1} that the class of $A$-uniform domains is a proper subset of the class of uniform domains  (see Subsection 2.2 for the definition) and thus a proper subset of the class of $\varphi$-uniform domains (for definition see Subsection 2.2). It is a natural question to consider whether or not there is an analogous result for uniform or $\varphi$-uniform domains as stated in Theorem \ref{z-1}. In particular,

{\it are uniform or $\varphi$-uniform domains preserved by an Apollonian bilipschitz mapping which is also quasiconformal?}

The main purpose of this paper is to deal with this question and we obtain that the uniformity, $\varphi$-uniformity and $\delta$-hyperbolicity (in the sense of Gromov with respect to quasihyperbolic metric, for definition see Subsection 2.6) of proper domains of $\mathbb{R}^n$ are invariant under roughly Apollonian bilipschitz mappings (see Subsection 2.11 for the definition) as follows.

\begin{theorem}\label{thm-1}
Let $G\subsetneq \mathbb{R}^n$ be a domain and let  $f:G\to G'\subsetneq\mathbb{R}^n$
be an  $(M,C)$-roughly Apollonian bilipschitz mapping. Then we have
\begin{enumerate}
\item If $G$ is $c$-uniform, then $G'$ is $c_1$-uniform  with $c_1$ depending only on $c$, $n$, $C$ and $M$;
\item If $G$ is $\varphi$-uniform, then $G'$ is $\varphi'$-uniform with $\varphi'$ depending only on $\varphi$, $n$, $C$ and $M$;
\item If $G$ is $\delta$-hyperbolic, then $G'$ is $\delta'$-hyperbolic with $\delta'$ depending only on $\delta$, $n$, $C$ and $M$.
\end{enumerate}
\end{theorem}

We remark that in Theorem \ref{thm-1} the quasiconformality for the maps is not needed. Next, as an application of Theorem \ref{thm-1} we shall demonstrate four equivalence conditions for a quasiconformal mapping which is defined on a uniform domain  to be roughly Apollonian bilipschitz.

%

\begin{theorem}\label{thm-2}
Let $G\subsetneq \mathbb{R}^n$ be a uniform domain and let $f:G\to G'\subsetneq \mathbb{R}^n$ be a quasiconformal mapping. Then the following conditions are equivalent:
\begin{enumerate}
  \item $f$ is  a roughly Apollonian bilipschitz mapping in $G$;
  \item $G'$ is uniform;
  \item $f:\overline{G}\to \overline{G'}$ is a homeomorphism and $f|_{\partial G}$ is quasim\"{o}bius;
  \item $f$ is quasim\"{o}bius in $G$.
\end{enumerate}
\end{theorem}

Moreover,   one can obtain the following invariance of $\varphi$-uniformity of domains in $\mathbb{R}^n$ under quasim\"obius mappings. Recently,  H\"{a}st\"{o}, Kl\'{e}n, Sahoo and Vuorinen \cite{HKSV}  studied the geometric properties of $\varphi$-uniform domains in $\mathbb{R}^n$. They proved that $\varphi$-uniform domains are preserved under quasiconformal mappings of $\mathbb{R}^n$. We restate this result in a stronger form which is more practical to check as follows.

\begin{theorem}\label{thm-3}Let $G\subsetneq \mathbb{R}^n$ be a $\varphi$-uniform domain and let $f:G\to G'\subsetneq\mathbb{R}^n$ be a $\theta$-quasim\"{o}bius homeomorphism. Then $G'$ is $\varphi'$-uniform with $\varphi'$ depending only on $\varphi$, $\theta$ and $n$.
\end{theorem}

The rest of this paper is organized as follows.  In Section \ref{sec-2}, we recall some definitions and preliminary results.
Section \ref{sec-3} is devoted to the proofs of our main results.
\section{Preliminaries}\label{sec-2}
\begin{nonsec}{\bf Notation.\,}We denote by $\mathbb{R}^n$ the Euclidean $n$-space and by $\overline{\mathbb{R}^n}= \mathbb{R}^n\cup \{\infty\}$ the one point compactification
of $\mathbb{R}^n$, $G$ and $G'$ are  proper domains in $\mathbb{R}^n$.\end{nonsec}
\begin{nonsec}{\bf Uniform domains.\,}
In 1979, uniform domains were introduced by Martio and Sarvas \cite{MS}. A domain $G\subsetneq \mathbb{R}^n$ is called {\it uniform} provided there exists a constant $c$
with the property that each pair of points $x$, $y$ $\in G$ can
be joined by a rectifiable curve $\gamma$ in $G$ satisfying
\begin{enumerate}
\item $\ell(\gamma)\leq c\,|x-y|$, and
\item $\min\{\ell(\gamma[x,z]),\ell(\gamma[z,y])\}\leq c\,d_G(z)$ for all $z\in \gamma$,
\end{enumerate}
\noindent where $d_G(z)=\dist(z,\partial G)$, $\ell(\gamma)$ denotes the arc length of $\gamma$,
$\gamma[x,z]$ the part of $\gamma$ between $x$ and $z$. \end{nonsec}There is an important characterization of uniform domains in terms of an
inequality for {\it$j$-metric}
$$j_G(x,y)=\log\Big(1+\frac{|x-y|}{\min\{d_G(x),d_G(y)\}}\Big)$$
and the {\it quasi-hyperbolic metric}
$$k_G(x,y)=\inf \int_\gamma \frac{|dx|}{d_G(x)},$$
where the infimum is taken over all rectifiable curves joining $x$ and $y$ in $G$.

\begin{theorem}\label{lem-0}$($\cite[Theorem 1]{GO}$)$ A domain $G \subset \mathbb{R}^n$  is uniform if and only if there exist constants $c$ and $d$ such that for all $ x,y\in G$
$$k_G(x,y)\leq c j_G(x,y)+d.$$
\end{theorem}
This form of the definition for uniform domains is due to Gehring and
Osgood \cite{GO} and subsequently, it was shown by Vuorinen \cite[$2.50(2)$]{Vu2} that the additive constant can be chosen to be zero.
 This observation leads to the definition of $\varphi$-uniform domains introduced in \cite{Vu2}. Let $\varphi:[0,\infty)\to [0,\infty)$ be a homeomorphism. A domain $G\subsetneq \mathbb{R}^n$ is called $\varphi$-{\it uniform} if for all $x$, $y$ in $G$
$$k_G(x,y)\leq \varphi(r_G(x,y))\;\;\;\;\;\mbox{where}\;\;\;\;\;r_G(x,y)=\frac{|x-y|}{\min\{d_G(x),d_G(y)\}}.$$

In order to give a simple criterion for $\varphi$-uniform domains,
consider domains  $G$ satisfying the following property \cite[Examples~2.50~(1)]{Vu2}:
there exists a constant $C\ge 1$ such that each pair of points
$x,y\in G$ can be joined by a rectifiable path $\gamma\in G$ with
$\ell(\gamma)\le C\,|x-y|$ and $\min\{d_G(x),d_G(y)\}\le
C\,d(\gamma,\partial G)$. Then $G$ is $\varphi$-uniform with
$\varphi(t)=C^2t$. In particular, every convex domain is
$\varphi$-uniform with $\varphi(t)=t$. However, in general, convex
domains need not be uniform.

\begin{nonsec}{\bf Natural domains.} Suppose that $\emptyset\not=A\subset G\subsetneq \mathbb{R}^n$. We write $$r_G(A)=\sup\{r_G(x,y): x\in A, y\in A\}.$$\end{nonsec}
Clearly,
$$   \frac{d(A)}{2d(A, \partial G)}  \le r_G(A) \le   \frac{d(A)}{d(A, \partial G)} \,, $$ where $d(A)$ denotes the diameter of set $A$ and $d(A, \partial G)$ is the distance from  set $A$ to the boundary $\partial G$.

 Let $\psi:[0,\infty)\to [0,\infty)$ be an increasing function.
 A domain $G \subsetneq \mathbb{R}^n$ is called $\psi$-{\it natural} if
$$k_G(A)\leq \psi(r_G(A))$$
for every nonempty connected set $A\subset G$ with $r_G(A)<\infty$, where $k_G(A)$ denotes the quasihyperbolic diameter of  set $A$.

We note that a $\varphi$-uniform domain is $\varphi$-natural, and every convex domain is $\psi$-natural with $\psi(t)=t$ (see, \cite[Theorems 2.8 and 2.9]{Vai6'}).
In fact, the next result from \cite{V1} shows that the class of natural domains is fairly large. Note that the growth of the function $\psi_n(t)$ in Lemma
\ref{lem-2} is $\approx t^n\,.$

\begin{lemma}\label{lem-2}$($\cite[Corollary 2.18]{V1}$)$
Every proper domain in $\mathbb{R}^n$ is $\psi_n$-natural with $\psi_n$ depending only on $n$.
\end{lemma}

 It should be noted that Lemma \ref{lem-2} is only valid in the finite dimensional case.
 In an infinite dimensional Hilbert space, the broken tube construction  in \cite[2.3]{Vai2004} provides an example of a domain, which is not natural.

\begin{nonsec}{\bf Gromov hyperbolic domains.\,} A geodesic metric space $X$ is called {\it $\delta$-hyperbolic}, $\delta\geq0$, if for all triples of geodesics $[x,y]$, $[y,z]$, $[z,x]$ in $X$ every point in $[x,y]$ is within distance $\delta$ from $[y,z]\cup[z,x]$. The property is often expressed by saying that geodesic triangles in $X$ are $\delta$-thin. In general, we say that a space is {\it Gromov hyperbolic} if it is $\delta$-hyperbolic for some $\delta$.\end{nonsec}

 We shall use the term {\it Gromov hyperbolic domain ($\delta$-hyperbolic)} for those proper domains in $\mathbb{R}^n$ that are Gromov hyperbolic in the quasihyperbolic metric.

\begin{example}The real line is $0$-hyperbolic. A classical example of a hyperbolic space is
the Poincar\'e half space $x_n > 0$ in $\mathbb{R}^n$ with the hyperbolic metric defined by the element of
length $\frac{|dx|}{x_n}$. This space is $\delta$-hyperbolic with $\delta = \log 3 $ \cite{CDP90}. More generally, uniform
domains in $\mathbb{R}^n$ with the quasihyperbolic metric are hyperbolic \cite[Theorem 1.11]{BHK}.

Some examples of nonhyperbolic domains are: (1) $G=\mathbb{R}^2\setminus \{ne_1:n\in \mathbf{Z}\}$; (2) $G=\{x\in \mathbb{R}^3: 0<x_3<1 \}$ \cite[2.11]{Vai05}.

\end{example}

\begin{nonsec}{\bf Quasim\"obius mapping.\,}
Let $X$ and $Y$ be metric spaces.
A quadruple in a space $X$ is  an ordered sequence $Q=(a,b,c,d)$ of four
distinct points in $X$. The cross ratio of $Q$ is defined to be the
number
$$\tau(Q)=|a,b,c,d|=\frac{|a-c|}{|a-d|}\cdot\frac{|b-d|}{|b-c|}.$$ Observe that the definition is extended in
the well known manner to the case where one of the points is
$\infty$. For example,
$$|a,b,c,\infty|= \frac{|a-c|}{|b-c|}.$$ If $X_0 \subset \dot{X}=X\cup \{\infty\}$ and if $f: X_0\to \dot{Y}=Y\cup \{\infty\}$
is an injective map, the image of a quadruple $Q$ in $X_0$ is the
quadruple $fQ=(fa,fb,fc,fd)$. \end{nonsec}

\begin{definition} \label{def2'} Let  $\theta: [0, \infty)\to [0, \infty)$ be a
homeomorphism.   An embedding $f: X_0\to
\dot{Y}$ is said to be {\it $\theta$-quasim\"obius},
or briefly $\theta$-$QM$, if the inequality $\tau(f(Q))\leq
\theta(\tau(Q))$ holds for each quadruple in $X_0$. In particular, if $\theta(t)=C\max\{t^{\lambda},t^{1/\lambda}\}$, then we say that $f$ is power quasim\"obius.\end{definition}

\begin{nonsec}{\bf Remark.\,}\label{rem1}$($\cite{Vai2}$)$ We remark that the inverse map $f^{-1}$ of a $\theta$-QM is $\theta'$-QM with  $\theta'(t)=\theta^{-1}(t^{-1})^{-1}$ for $t>0$. If $f:A\to \dot{Y}$ is $\theta_1$-QM and $g:f(A)\to \dot{Z}$ is $\theta_2$-QM, then the composition $g\circ f$ is $\theta$-QM with $\theta(t)=\theta_2(\theta_1(t))$.
If $f$ is $\theta$-QM with $\theta(t)=t$, then we say that $f$ is a M\"obius map.  In particular, the inversion $u$ defined by $u(x)=\frac{x}{|x|^2}$ is M\"obius in an inner product space. \end{nonsec}

\begin{nonsec}{\bf Roughly  bilipschitz mappings and  quasiconformal mappings.\, }

A homeomorphism  $f: (G,m_G)\to (G',m_{G'})$ is said to be an {\it $M$-roughly $C$-bilipschitz in the $m$ metric}, if $M\geq 1$, $C\geq 0$, and $$\frac{m_G(x,y)-C}{M}\leq m_{G'}(f(x),f(y))\leq M m_G(x,y)+C$$ for all $x,y\in G$.
 A homeomorphism  $f: G\to G'$ is said to be an {\it $(M,C)$-roughly Apollonian bilipschitz}, if it is $M$-roughly $C$-bilipschitz in the Apollonian metric. This means that $f$ is a homeomorphism such that  $$\frac{\alpha_G(x,y)-C}{M}\leq \alpha_{G'}(f(x),f(y))\leq M\alpha_G(x,y)+C$$ for all $x,y\in G$.
 Similarly, we say that a homeomorphism $f$ is {\it $C$-coarsely $M$-quasihyperbolic}, abbreviated {\it $(M,C)$-CQH} if it is $M$-roughly $C$-bilipschitz in the quasihyperbolic metric. This means that $f$ is a homeomorphism such that $$\frac{k_G(x,y)-C}{M}\leq k_{G'}(f(x),f(y))\leq Mk_G(x,y)+C$$ for all $x,y\in G$.\end{nonsec}

 The basic theory of quasiconformal mappings in $\mathbb{R}^n, n\geq 2$ is given in V\"ais\"al\"a's book \cite{Vaibook}. There are plenty of mutually equivalent definitions for quasiconformality in $\mathbb{R}^n$. In this paper we adopt the following simplified version of the metric definition. Let $n\geq 2$, let $G$ and $G'$ be domains in $\mathbb{R}^n$, and let 
 $f:G\to G'$ be a homeomorphism. For $x\in G$,  The {\it linear dilatation} of $f$ at $x\in G$ is defined by $$H_f(x):=\limsup_{r\rightarrow 0}\frac{\sup\{|f(x)-f(y)|:|x-y|=r\}}{\inf\{|f(x)-f(z)|:|x-z|=r\}}.$$
For $1\leq K< \infty$, we say that $f:G\to G'$ is $K$-{\it quasiconformal} if $H_{f}(x)\leq K $ for all $ x\in G$, and that $f$ is {\it quasiconformal} if it is $K$-quasiconformal for some $K$.

For a $K$-quasiconformal mapping we have the following property.

\begin{lemma}\label{lem2}$($\cite[Theorem $3$]{GO}$)$
For $n\geq 2$, $K\geq 1$, there exist constants $c$ and $\mu$ depending only on $n$
and $K$ with the following property. If $G, G'\subset \mathbb{R}^n$ and
$f: G\to G'$ is a $K$-quasiconformal mapping, then for all $x,y\in G$,
$$k_{G'}(f(x),f(y))\leq c \max\{k_G(x,y), (k_G(x,y))^{\mu}\}.$$
\end{lemma}

The next result deals with the case when the mapping is defined in
$ \mathbb{R}^n\,.$

\begin{lemma}\label{lem2'}$($\cite[Lemma $2.3$]{HKSV}$)$
For $n\geq 2$, $K\geq 1$, there exist constants $c$ and $\mu$ depending only on $n$ and $K$  with the following property. If $f: \mathbb{R}^n\to \mathbb{R}^n$ is a $K$-quasiconformal mapping,
$G, G'\subset \mathbb{R}^n$ are domains, and  $fG= G'$, then for all $x,y\in G$,
$$j_{G'}(f(x),f(y))\leq c \max\{j_G(x,y), (j_G(x,y))^{\mu}\}.$$
\end{lemma}

\begin{nonsec}{\bf Remark.\,}Let $G_1$ and $G_2$ be proper domains of $\mathbb{R}^n$. We know from Lemma \ref{lem-2} that $G_i$ $(i=1,2)$  is $\psi_i$-natural with $\psi_i$ depending only on $n$. Suppose that  $f:G_1\to G_2$ is a $K$-quasiconformal mapping of $\mathbb{R}^n$ which maps $G_1$ onto $G_2$, then  we see from Lemmas \ref{lem2} and \ref{lem2'} that $G_2$ is  $\psi_2$-natural with $\psi_2=\psi_2(\psi_1,n,K)$.
\end{nonsec}

Moreover, we see from Lemma \ref{lem2} and \cite[Theorem $4.14$]{Vai0} that a quasiconformal mapping is CQH, which we state as follows.

\begin{lemma}\label{lem1}$($\cite[Theorem $3$]{GO} {\rm and} \cite[Theorem $4.14$]{Vai0}$)$
For $n\geq 2$, $K\geq 1$, there exist constants $M\geq 1$, $C>0$ such that if $G, G'\subset \mathbb{R}^n$ and  $f: G\to G'$ is $K$-quasiconformal, then  $f$ is $(M,C)$-CQH.
\end{lemma}

\begin{nonsec}{\bf Remark.\,}\label{rem2}
Let $G$ be a proper domain of $\mathbb{R}^n$. We consider the Apollonian metric $\alpha_G$, Seittenranta's metric $\delta_G$
which is defined as $($\cite{S}$)$
$$\delta_G(x,y)=\log(1+\sup_{a,b\in\partial G}|a,x,b,y|),$$
the metric $h_{G,c}$ $(c\geq 2)$ which is defined as $($\cite{DHV}$)$
$$h_{G,c}(x,y)=\log(1+c\frac{|x-y|}{\sqrt{d_G(x)d_G(y)}})$$
and $j_G$ metric. We see from \cite[Theorems 3.4 and 3.11]{S} and  \cite[Lemma 4.4]{DHV} that the inequalities \begin{equation}\label{rem-eq-1}j_G\leq \delta_G\leq 2j_G,\end{equation} \begin{equation}\label{rem-eq-2}\alpha_G\leq \delta_G\leq \log(e^{\alpha_G}+2)\leq \alpha_G+\log 3\end{equation} and $$\frac{c}{2(1+c)}j_G\leq h_{G,c}\leq cj_G$$ hold for every proper domain $G$ of $\mathbb{R}^n$. Hence, the identity  map $id: (G, m_1)\to (G, m_2)$ is roughly bilipschitz, where $m_1,m_2\in \{\alpha_G,j_G, \delta_G,h_{G,c}\}$. \end{nonsec}

\begin{nonsec}{\bf Quasi-isotropic.\, }
 The concept of quasi-isotropy which is a kind of local comparison property was introduced by H\"ast\"o in \cite{H1}, and was the focus of \cite{H2}. Let $G \subsetneq \mathbb{R}^n$. We recall  that a metric space $(G,d)$ is {\it $L$-quasi-isotropic} ($L\geq 1$) \cite{H1} if $$\limsup_{r\to 0}\frac{\sup\{d(x,z):|x-z|=r\}}{\inf\{d(x,y):|x-y|=r\}}\leq L $$ for every $x\in G$, where $|x-z|$ means the Euclidean distance of $x$ and $z$. In this paper, we say a domain $G$ is  $L$-quasi-isotropic if $(G,\alpha_G)$ is $L$-quasi-isotropic.
\end{nonsec}

%

\section{The proofs of main results}\label{sec-3}
\begin{nonsec}{\bf Basic lemmas.\,} In this section, we shall give the proofs of our main results. We first introduce some basic inequalities which are important to our proofs.\end{nonsec}

\begin{lemma}\label{lem-3}
Let $G \subsetneq \mathbb{R}^n$ be a domain.
\begin{enumerate}
\item $($\cite[Theorems 3.4 and 3.11]{S}$)$ For all $x,y\in G$, $$j_G(x,y)\leq \alpha_G(x,y)+\log 3;\; \frac{1}{ 2}\alpha_G(x,y)\leq j_G(x,y)\leq k_G(x,y);$$

\item $($\cite[Theorem 3.9]{Vai1}$)$ If $|x-y|\leq d_G(x)/2$ or $k_G(x,y)\leq 1$ with $x,y\in G$, then
$$\frac{|x-y|}{2d_G(x)}\leq k_G(x,y)\leq \frac{2|x-y|}{d_G(x)}.$$
\item If $G$ is  a $c$-uniform domain, then  we have $$k_G(x,y)\leq c_1 j_G(x,y)\leq c_1(\alpha_G(x,y)+\log 3),$$ where $c_1=c_1(c)$. Moreover, for the uniform domain $G=\mathbb{R}^n\setminus \{0\}$ we note that there does not exist any constant $c_2\geq 0$ such that $k_G(x,y)\leq c_2\alpha_G(x,y)$ holds for all $x,y\in G$.
\end{enumerate}
\end{lemma}

P. H\"ast\"o \cite{H1} investigated domains $G$ for which $\alpha_G(x,y) \ge K  j_G(x,y)\,.$
He found a sufficient condition on the domain under which this holds, in particular this
sufficient condition fails for $\mathbb{R}^n \setminus \{0\}$ and requires that the boundary of
the domain is "thick".

\begin{lemma}\label{lem-4}
Let $G, G'\subsetneq \mathbb{R}^n$ be domains and let a homeomorphism $f:G\to G'$ be an $(M,C)$-roughly Apollonian bilipschitz mapping. Then $f:G\to G'$ is an $(M',C')$-CQH with $(M',C')$ depending on $M, C$ and $n$ only.
\end{lemma}

\begin{proof} We may assume that there are constants $M\geq 1$ and $C\geq 0$ such that $f:G\to G'$ is $(M,C)$-roughly Apollonian bilipschitz. Thanks to \cite[Lemma $2.3$]{Vai1} and by symmetry, we only need to estimate $k_{G'}(f(x),f(y))$ for all $x,y\in G$ with $k(x,y)\leq \frac{1}{8}$, because $(G,k_G)$ and $(G',k_{G'})$ are geodesic metric spaces and evidently $c$-quasi-convex with $c=1$.

Towards this end, first by Lemma \ref{lem-3}, we have $$r_G(x,y)=\frac{|x-y|}{\min\{d_G(x),d_G(y)\}}\leq 2 k_G(x,y)\leq \frac{1}{4},$$ and so the segment $A=[x,y]\subset G$. Moreover, we have
$$r_G(A)\leq \frac{\diam(A)}{\dist(A,\partial G)}\leq \frac{r_G(x,y)}{1-r_G(x,y)}\leq 1/2.$$
Then for all $a,b\in A$, we get $$r_G(a,b)\leq r_G(A)\leq 1/2,$$ which, together with Lemma \ref{lem-3}, implies that
$$\alpha_G(a,b)\leq 2j_G(a,b)=2\log(1+r_G(a,b))< 2\log 2.$$
Furthermore, on one hand, since $f:G\to G'$ is $(M,C)$--roughly Apollonian bilipschitz, we have
$$\alpha_{G'}(f(a),f(b))\leq M\alpha_G(a,b)+C< 2M\log 2+C,$$
On the other hand, again by using Lemma \ref{lem-3}, we obtain
$$r_{G'}(f(a),f(b))= e^{j_{G'}(f(a),f(b))}-1\leq e^{2M\log 2+C+\log3}-1=:L,$$
so $r_{G'}(f(A))\leq L$.
Hence we see from Lemma \ref{lem-2} that there is an increasing function $\psi_n:[0,\infty)\to [0,\infty)$ such that
$$k_{G'}(f(x),f(y))\leq k_{G'}(f(A))\leq \psi_n(r_{G'}(f(A))\leq \psi_n(L).$$
The proof is complete.
\end{proof}

\begin{lemma}\label{lem-4-1}
Suppose that $f:G\to G'$ is a $\theta$-quasim\"obius homeomorphism between two proper domains of $\mathbb{R}^n$, then $f$ is an $(M,C)$-roughly Apollonian bilipschitz mapping with $M$, $C$ depending only on $\theta$.
\end{lemma}
\begin{proof}  We first observe from \cite[Theorem 3.19]{Vai2} that $f$ has a quasim\"obius extension $\overline{f}: \overline{G}\to \overline{G'}$. To show that $f$ is roughly Apollonian bilipschitz, we only need to prove that $f$ is power quasim\"obius, that is, there exist constants $C\geq 1$ and $\lambda\geq 1$ depending only on $\theta$ such that $f$ is $\theta_1$-QM  with $\theta_1(t)=C\max\{t^{\lambda},t^{1/\lambda}\}$. Indeed, this can be seen as follows. For any $x,y \in G$ and $a,b\in\partial G$, we note that
$$|a,y,x,b|=|b,y,x,a|^{-1}\;\;\; \mbox{and}\;\;\; \alpha_G(x,y)=\log\sup_{a,b\in \partial G}|a,y,x,b|.$$
Without loss of generality we may assume that $|a,y,x,b|\geq1$. Since $f$ is $\theta_1$-QM with $\theta_1(t)=C\max\{t^{\lambda},t^{1/\lambda}\}$, we have
$$\log|f(a),f(y),f(x),f(b)|\leq \lambda \log |a,y,x,b|+\log C \leq\lambda \alpha_G(x,y)+\log C,$$
so by the arbitrariness of $a,b\in \partial G$, we get
$$\alpha_{G'}(f(x),f(y))\leq \lambda\alpha_G(x,y)+\log C.$$
Since the inverse map of power quasim\"obius is also power quasim\"obius, by symmetry, the assertion follows.

To this end, by auxiliary translations we may assume that $0\in\partial G$ and that either $\overline{f}(0)=0$ or $\overline{f}(0)=\infty$. Let $u$ be the inversion $u(x)=\frac{x}{|x|^2}$. We  note from Remark \ref{rem1} that $u$ is M\"obius. If $\overline{f}(0)=0$, we define $g:u(G)\to u(G')$ by $g(x)=u\circ f\circ u(x)$. If $\overline{f}(0)=\infty$, we define  $g:u(G)\to G'$ by $g(x)=f\circ u(x)$. In both cases, we have that $g$ is $\theta$-QM. Since $g(x)\rightarrow \infty$ as $x\rightarrow\infty$, $g$ is $\theta$-QS, see \cite[Theorem 3.10]{Vai2}. Moreover, by \cite[Corollary 3,12]{TV} we have that $g$ is $\theta_1$-QS with $\theta_1(t)=C\max\{t^{\lambda},t^{1/\lambda}\}$, where $C\geq 1$ and $\lambda\geq 1$ depend only on $\theta$. Hence, we get that $f$ is $\theta_1$-QM with $\theta_1(t)=C\max\{t^{\lambda},t^{1/\lambda}\}$ as desired.

Hence the proof of this Lemma is complete.\end{proof}


\begin{nonsec}{\bf The proof of Theorem \ref{cor}.\,}
$(1)\Rightarrow(2):$ This implication follows from \cite[Proposition 6.6]{H1}.\end{nonsec}

$(2)\Rightarrow(3):$ It follows from the assumption, \cite[Corollary $5.4$]{H1} and \cite[Proposition 6.6]{H1} that there is $M\geq 1$ such that $f$ is $M$-bilipschitz with respect to the quasihyperbolic metrics. Let $x_0\in G$ and $r\in(0,\frac{d_G(x_0)}{2M})$. For all $x,y\in \mathbb{S}^{n-1}(x_0,r)$, we have $k_G(x,x_0)\leq \frac{1}{M}$ by means of \cite[Theorem $3.9$]{Vai1}, and so $k_{G'}(f(x),f(x_0))\leq 1$. Again by \cite[Theorem $3.9$]{Vai1}, we obtain that
$$\frac{|f(x)-f(x_0)|}{2d_{G'}(f(x_0))}\leq k_{G'}(f(x),f(x_0))\leq \frac{2|f(x)-f(x_0)|}{d_{G'}(f(x_0))}.$$
Hence we have
\begin{eqnarray*}\limsup_{r\to 0+}\frac{|f(x)-f(x_0)|}{|f(y)-f(x_0)|}&\leq& \limsup_{r\to 0+}\frac{4k_{G'}(f(x),f(x_0))}{k_{G'}(f(y),f(x_0))}\\\nonumber&\leq& \limsup_{r\to 0+}\frac{4M^2 k_G(x,x_0)}{k_G(y,x_0)}\leq 16M^2,\end{eqnarray*}
as desired.

$(3)\Rightarrow(1):$ This implication follows from Theorem \ref{z-1}.

$(3)\Rightarrow(4):$ Assume that $f$ is quasiconformal in $G$, then Theorem \ref{z-1} yields that $f(G)$ is $A$-uniform. Hence, we get from the fact ``an $A$-uniform domain is uniform" and \cite[Theorem 5.6]{Vai2} that $f$ is quasim\"obius, as desired.

$(4)\Rightarrow(3):$ This implication follows from \cite[Theorem $5.2$]{Vai2}.

\begin{nonsec}{\bf The proof of Theorem \ref{thm-1}.\,}
 Let $f:G\to G'$ be roughly Apollonian bilipschitz. Then by Lemma \ref{lem-4} we may assume that $f$ is $(M,C)$-roughly Apollonian bilipschitz and  $(M,C)$-CQH for some constants $M\geq 1$ and $C\geq 0$. Hence, we see from  Lemma \ref{lem-3}  that  \begin{equation}\label{eq-thm1-1}j_G(x,y)\leq  2Mj_{G'}(f(x),f(y))+C+\log 3.\end{equation}  We first prove part (1), that is, if $G$ is uniform, then $G'$ is uniform. Suppose that $G$ is $c$-uniform for some constant $c\geq 1$. According to Theorem \ref{lem-0}, there exist positive constants $c_1$ and $c_2$ such that $$k_G(x,y)\leq c_1j_G(x,y)+c_2$$ for all $x,y\in G$. One computes from these facts that
\begin{eqnarray*} k_{G'}(f(x),f(y)) &\leq& Mk_G(x,y)+C
\\&\leq& M[c_1j_G(x,y)+c_2]+C
\\&\leq& c_1M[Mj_{G'}(f(x),f(y))+C+\log3]+c_2M+C.\end{eqnarray*}
Again by Theorem \ref{lem-0}, we immediately see that $G'=f(G)$ is uniform. Hence, part (1) holds.\end{nonsec}

Next, we prove part (2). Assume that $G$ is $\varphi$-uniform,  to prove $G'$ is $\varphi'$-uniform, we only need to find a homeomorphism $\varphi':[0,\infty)\to [0,\infty)$ such that $$k_{G'}(f(x),f(y))\leq \varphi'(r_{G'}(f(x),f(y)))$$ for all $x,y\in G$.
To this end, we divide the proof into two cases.

{\em Case A.\,} $|f(x)-f(y)|\leq \frac{1}{2}\min\{d_{G'}(f(x)),d_{G'}(f(y))\}$.

Then by Lemma \ref{lem-3} we have
$$k_{G'}(f(x),f(y))\leq 2\frac{|f(x)-f(y)|}{d_{G'}(f(x))}\leq 2r_{G'}(f(x),f(y)),$$ which give the desired $\varphi'$ with $\varphi'(t)=2t$.

{\em Case B.\,} $|f(x)-f(y)|> \frac{1}{2}\min\{d_{G'}(f(x)),d_{G'}(f(y))\}$.
Then $$j_{G'}(f(x),f(y))=\log(1+\frac{|f(x)-f(y)|}{\min\{d_{G'}(f(x)),d_{G'}(f(y))\}})
>\log\frac{3}{2}.$$ Let $\varphi_1(t)=\varphi(e^t-1)$. Then we see from \eqref{eq-thm1-1} that
\begin{eqnarray*} k_{G'}(f(x),f(y)) &\leq& M k_G(x,y)+C
 \leq  M\varphi_1(j_G(x,y))+C
 \\&\leq& M\varphi_1(2Mj_{G'}(f(x),f(y))+C+\log 3)))+C\\&\leq& M\varphi_1\left((2M+\frac{C+\log3}{\log\frac{3}{2}})j_{G'}(f(x),f(y))\right)\\&+&\frac{C}{\log\frac{3}{2}}j_{G'}(f(x),f(y)).\end{eqnarray*}
By letting $\varphi'(t)=M\varphi_1\left((2M+\frac{C+\log3}{\log\frac{3}{2}})\log(1+t)\right)+\frac{C}{\log\frac{3}{2}}\log(1+t)$, we complete the proof in this case.

Combining Case A and Case B, we complete the proof of part (2).

Finally, we come to prove part (3). It follows from Lemma \ref{lem-4} and the fact that a Gromov hyperbolic domain under a CQH homeomorphism  is still Gromov hyperbolic, see \cite{BHK} (or \cite[Theorem 3.18]{Vai10}).\qed

\begin{nonsec}{\bf Remark.\,}
Let $G$ be a proper domain of $\mathbb{R}^n$. We consider the Apollonian metric $\alpha_G$, Seittenranta's metric $\delta_G$,
the metric $h_{G,c}$ $(c\geq 2)$
and $j_G$ metric. We see from Remark \ref{rem2} and the proof of  Lemma \ref{lem-4} that if we replace the Apollonian metric by  $m_G\in\{j_G, \delta_G,h_{G,c}\}$, then we have the following holds: {\it If $f: (G, m_G)\to (G', m_{G'})$ is an $M$-roughly $C$-bilipschitz mapping, then $f:G\to G'$ is an $(M',C')$-CQH with $(M',C')$ depending on $M, C$ and $n$ only.}  Hence, we easily see that Theorem \ref{thm-1} is also true if we replace the Apollonian metric by  $m_G\in\{j_G, \delta_G,h_{G,c}\}$.\\

\end{nonsec}


\begin{nonsec}{\bf The proof of Theorem \ref{thm-2}.\,} The equivalence of $(1)\Rightarrow(2)\Rightarrow(4)\Rightarrow(1)$ are from Theorem \ref{thm-1}, \cite[Theorem 5.6]{Vai2} and Lemma \ref{lem-4-1}. It remains to show $(3)\Leftrightarrow(4)$.\end{nonsec}

The implication $(4)\Rightarrow(3)$ follows from \cite[Theorem 3.19]{Vai2};

$(3)\Rightarrow(4):$ this implication follows from \cite[Theorem $3.15$]{Vai-3}.

\begin{nonsec}{\bf The proof of Theorem \ref{thm-3}.\,}The proof  follows from  Theorem \ref{thm-1} and  Lemma \ref{lem-4-1}.\end{nonsec}

\end{document}